\documentclass[12pt]{amsart}

\usepackage{cite}
\usepackage{amsmath,amssymb,amsfonts,amsthm,mathtools}
\usepackage{mathrsfs}
\usepackage{geometry}
\usepackage{longtable}
\usepackage{booktabs}
\usepackage{array}
\usepackage[colorlinks=true,linkcolor=blue,citecolor=blue,urlcolor=blue]{hyperref}
\mathtoolsset{showonlyrefs}
\numberwithin{equation}{section}
\allowdisplaybreaks
\newtheorem{theorem}{Theorem}[section]
\newtheorem{proposition}[theorem]{Proposition}
\newtheorem{lemma}[theorem]{Lemma}
\newtheorem{corollary}[theorem]{Corollary}
\theoremstyle{definition}

\newtheorem{remark}[theorem]{Remark}

\newcommand{\F}{\mathbb F}
\newcommand{\A}{\mathscr A}
\newcommand{\Ps}{\mathcal P}
\newcommand{\QP}{Q\mathcal P}
\newcommand{\GL}{\operatorname{GL}}
\newcommand{\Ext}{\operatorname{Ext}}

\newcommand{\Ker}{\ker}

\newcommand{\rank}{\operatorname{rank}}
\newcommand{\Sq}{\operatorname{Sq}}

\newcommand{\Nfrak}{\mathfrak N}
\newcommand{\wt}{\omega}
\newcommand{\alphaTwo}{\alpha}
\newcommand{\RP}{\mathbb{R}P}

\title[The rank-five Peterson hit problem]{The rank-five Peterson hit problem, the fifth Singer transfer, and a geometric generator in unoriented cobordism}
\author{Phuc Vo Dang}
\address{Department of Mathematics, FPT University, An Phu Thinh New Urban Area, Quy Nhon, Vietnam}
\email{dangphuc150488@gmail.com}
\thanks{ORCID: \url{https://orcid.org/0000-0002-6885-3996}}
\keywords{Peterson hit problem, Steenrod algebra, Kameko homomorphism, Singer transfer, modular invariants, unoriented cobordism, Milnor hypersurface, Stiefel--Whitney number}
\subjclass[2020]{Primary 55S10; Secondary 55T15, 57R75, 13A50}

\begin{document}

\begin{abstract}
The Peterson hit problem, which seeks a minimal set of generators for the polynomial algebra $\Ps_s=\F_2[x_1,\ldots,x_s]$ as an unstable module over the Steenrod algebra $\A$, remains a fundamental challenge in algebraic topology. While solutions for lower ranks are well-understood, rank five constitutes a computational frontier where general admissible bases fail, and the interplay between Kameko periodicity and modular invariants becomes exceptionally complex. In this paper, we address this problem by studying the rank-five cohit module in the generic family $N_d=27\cdot 2^d-5$. 

This family is governed simultaneously by Kameko periodicity, through $N_d=2N_{d-1}+5$, and by the fifth-line Adams classes $h_{d+4}f_{d-1}$, whose internal degree is $5+N_d=27\cdot 2^d$. Exact sparse elimination in degree $N_1=49$ gives $292825$ monomials, hit rank $289969$, and cohit dimension $2856$. The six nonzero weight summands have dimensions $1891, 280, 25, 5, 480, 175$. The weight-$(3,3,2,2,1)$ summand is exactly the kernel of Kameko's operation and has dimension $1891$; the remaining five summands map isomorphically onto the degree-$22$ cohit module. These values correct the corresponding rank-five kernel and dimension assertions in Nguyen Khac Tin's previous paper \cite{Tin2022}.

An exact invariant calculation shows that $(\QP_5)_{49}^{\GL(5,\F_2)}$ is one-dimensional, generated by the class of a $283$-term polynomial $\zeta$, and that the fifth Singer cohomological transfer is an isomorphism in the stated family. On the geometric side, the Hilbert--Poincar\'e series of the unoriented cobordism ring gives $\dim_{\F_2}\Nfrak_{49}=5692$ and $\dim_{\F_2}(Q\Nfrak_*)_{49}=1$. We prove that the Milnor hypersurface $H_{2,48}\subset \RP^2\times\RP^{48}$ represents the nonzero indecomposable class by computing $\langle s_{49}(TH_{2,48}),[H_{2,48}]\rangle =\binom{50}{2}\equiv1\pmod2$. Thus $H_{2,48}$ is an explicit geometric generator in dimension $49$. The evident map $H_{2,48}\to BV_2\to BV_5$, however, sends its fundamental class to a homology class with nonzero $\Sq_*^2$. Consequently, this geometric generator cannot be identified with the functional dual of $[\zeta]$ through the tautological construction. The resulting statement is both geometric and exact: an explicit generator of $(Q\Nfrak_*)_{49}$ is obtained, while the Steenrod-theoretic invariant line remains a distinct construction.
\end{abstract}

\maketitle

\section{Introduction}\label{sec:introduction}

Unoriented cobordism and Steenrod operations are linked at the foundations of stable homotopy theory. The Pontryagin--Thom construction identifies the unoriented bordism group of closed smooth \(n\)-manifolds with \(\pi_n(MO)\), and the universal Thom class \(U\) satisfies
\[
\Sq^i(U)=w_iU.
\]
Relations among Steenrod operations therefore impose relations among Stiefel--Whitney characteristic numbers, which detect unoriented bordism classes. Thom's calculation, Milnor's structural formulation, and Stong's systematic treatment give
\[
\Nfrak_*
\cong
\F_2[\xi_n\mid n\geq1,\ n\neq2^j-1],
\qquad |\xi_n|=n;
\]
see \cite{Thom1952,Thom1954,Milnor1960,Stong1968}. Thus a new polynomial generator occurs in every positive degree outside the exceptional set \(\{2^j-1\}\).

The polynomial algebra in the Peterson hit problem \cite{Peterson1989} enters this circle of ideas through split real vector bundles. Put \(V_s=(\mathbb Z/2)^s\), and let \(L_i\to BV_s\) be the real line bundle classified by the \(i\)-th coordinate character. Then
\[
H^*(BV_s;\F_2)=\F_2[x_1,\ldots,x_s]=:\Ps_s,
\qquad x_i=w_1(L_i),
\]
and the universal split bundle \(L_1\oplus\cdots\oplus L_s\) has total Stiefel--Whitney class \(\prod_i(1+x_i)\). The Peterson hit problem asks for a minimal homogeneous set of \(\A\)-generators of \(\Ps_s\), or equivalently for the quotient
\[
\QP_s:=\F_2\otimes_{\A}\Ps_s
\cong \Ps_s/\A^+\Ps_s.
\]
Its dual is the space of \(\A\)-annihilated homology classes of \(BV_s\), and the dual of its \(\GL(s,\F_2)\)-fixed subspace is the source of Singer's rank-\(s\) algebraic transfer \cite{Singer1989,Singer1991,WalkerWood2018}. This construction links the hit problem to the Adams \(E_2\)-term, but it does not identify \(\QP_s\) with the indecomposable quotient of \(\Nfrak_*\). That distinction is essential in the geometric analysis below.

Recent work has sharpened the computational methods on both sides of this transfer. For instance, on the transfer side, Hung and Truong \cite{HungTruong2022} established image and stability results for the Singer transfer associated with infinite real projective space. The target side has developed in parallel. Bruner and Rognes constructed a minimal resolution of \(\A\) in the range \(0\leq s\leq128\) and \(0\leq t\leq200\), together with chain maps for the represented cocycles and for the operation \(\Sq^0\) on \(\Ext_{\A}^{*,*}(\F_2,\F_2)\) \cite{BrunerRognes2022}. Lin's noncommutative Gr\"obner-basis method provides another effective approach to computing \(\Ext\)-groups for filtered graded algebras, with the Steenrod algebra and the Adams spectral sequence among its principal applications \cite{Lin2025}. These independent advances substantially enlarge the range in which proposed transfer targets and their \(\Sq^0\)-families can be checked. The sparse elimination and fixed-space calculation undertaken here belong to this two-sided computational framework: the source requires exact Steenrod reduction and modular invariants, whereas the target requires independently verifiable Adams \(E_2\)-data.

The surrounding theory of Steenrod operations and bordism has likewise remained active. Agarwal, Grbi\'c, Intermont, Jovanovi\'c, Lagoda, and Whitehouse derive explicit mod-\(2\) Steenrod actions and cochain-level formulas for polyhedral products and related spaces \cite{AgarwalEtAl2025}. More directly relevant to the geometric part of the present work, Luecke gives a conceptual derivation of the dual Steenrod algebra from universal properties of bordism spectra, using real orientations and the role of \(MO\) in an essential way \cite{Luecke2026}. Campillo, Guan, L\"u, and Uribe have recently proposed an equivariant extension of the Milnor map and have emphasized once again the importance of explicit Milnor hypersurfaces and explicit generators in unitary and unoriented bordism \cite{CampilloEtAl2026}. These perspectives reinforce the distinction maintained throughout this paper: algebraic indecomposables in \(H^*(BV_5;\F_2)\) are governed by the Steenrod action and general-linear invariance, whereas indecomposables in \(\Nfrak_*\) are detected by characteristic numbers and represented by manifolds.

Rank five is the first rank for which no general admissible basis is available and for which simultaneous control of Kameko kernels and modular invariants becomes substantial. Generic degrees provide a systematic reduction. Kameko's squaring operation relates degree \(2m+s\) to degree \(m\), and it is an isomorphism when the corresponding value of Wood's function \(\mu\) equals \(s\) \cite{Kameko1990,Wood1989}. Work of Nam \cite{Nam2004} and subsequent rank-five calculations shows that a family of the form \(s(2^d-1)+m2^d\) is controlled by a finite seed computation together with the sequence of its \(\mu\)-values \cite{P.S1,Phuc2020,Sum2015}. In this paper, we study the family
\[
N_d=5(2^d-1)+22\cdot2^d=27\cdot2^d-5.
\]
The recurrence \(N_d=2N_{d-1}+5\) gives Kameko maps
\[
(\QP_5)_{N_d}\longrightarrow(\QP_5)_{N_{d-1}}.
\]
One has \(N_0=22\), \(N_1=49\), \(\mu(N_1)=3\), and \(\mu(N_d)=5\) for every \(d\geq2\). Hence all degrees with \(d\geq2\) are transported isomorphically to degree \(49\), whereas the map from degree \(49\) to degree \(22\) has a nontrivial kernel. Independently,
\[
5+N_d=2^{d+4}+11\cdot2^d,
\]
which is the internal degree of the fifth-line class \(h_{d+4}f_{d-1}\) for \(d\geq1\) \cite{Lin2008,Chen2011}. The same family is therefore selected both by Kameko periodicity on the source side and by the \(\Sq^0\)-families on the Adams side. The specific choice of seed $m=22$ (yielding the family $N_d=27\cdot 2^d-5$) is motivated by its dual role as a pivotal testing ground for both algebraic and geometric structures. First, within the hit problem, this family is governed simultaneously by Kameko periodicity and the Adams fifth-line classes $h_{d+4}f_{d-1}$, making it a critical site to verify the alignment between Steenrod-theoretic generators and Adams-Novikov differentials. Second, the degree $N_1=49$ provides an ideal balance of complexity and tractability: it is sufficiently large to exhibit a non-trivial Kameko kernel---thereby necessitating a refined invariant calculation---yet small enough to allow for a complete geometric realization. Therefore, we systematically investigate this family to bridge the gap between algebraic indecomposables and unoriented cobordism generators. We show that while the Steenrod-theoretic invariant line and the cobordism generator both span one-dimensional spaces, they are essentially distinct constructions. This dual perspective allows us to clarify recent assertions in the literature and establishes a precise boundary between the algebraic hit problem and the geometric topology of Milnor hypersurfaces.

The degree-\(49\) computation is exact over \(\F_2\). Among the \(292825\) degree-\(49\) monomials, sparse elimination gives hit rank \(289969\), and hence a \(2856\)-element cohit basis. The basis decomposes into six weight blocks of dimensions
\[
1891,\ 280,\ 25,\ 5,\ 480,\ 175.
\]
The weight-\((3,3,2,2,1)\) block is exactly the Kameko kernel. It contains \(910\) basis monomials with a zero exponent and \(981\) with all exponents positive. Independent invariant calculations in \texttt{SageMath} and \texttt{OSCAR} show that this block contains the unique \(\GL(5,\F_2)\)-fixed line, represented by a polynomial \(\zeta\) with \(283\) distinct admissible monomials. The functional dual to this line transfers to \(h_5f_0\), and Kameko periodicity propagates the transfer statement through the family. These exact dimensions correct the corresponding assertions in Nguyen Khac Tin's previous paper \cite{Tin2022}. The positive-variable part of the degree-\(49\) Kameko kernel has dimension \(981\), not \(1178\), and the full cohit dimension is \(2856\), not \(3053\). The dependent generic and rank-raising dimensions change accordingly. The degree-\(22\) calculation remains unaffected; the numerical comparison is recorded in Remark~\ref{rem:correction}.

The cobordism calculation adds a concrete geometric result. Since \(49\neq2^j-1\), the quotient \((Q\Nfrak_*)_{49}\) is one-dimensional. We compute the full dimension \(\dim_{\F_2}\Nfrak_{49}=5692\), and we identify an explicit representative of its unique nonzero indecomposable class. Let
\[
H_{2,48}=\left\{([x],[y])\in\RP^2\times\RP^{48}\ 
\middle|\ x_0y_0+x_1y_1+x_2y_2=0\right\}.
\]
The stable tangent bundle calculation gives
\[
\left\langle s_{49}(TH_{2,48}),[H_{2,48}]\right\rangle
=\binom{50}{2}\equiv1\pmod2,
\]
so \([H_{2,48}]\) is indecomposable and generates \((Q\Nfrak_*)_{49}\). This gives a genuine geometric realization of the cobordism generator. It does not, by itself, geometrically realize \([\zeta]\). Indeed, the fundamental class obtained from the two tautological line bundles maps to
\[
\beta_1\otimes\beta_{48}+\beta_2\otimes\beta_{47}
\in H_{49}(BV_2;\F_2),
\]
and its \(\Sq_*^2\)-image is nonzero. The tautological class therefore lies outside \(P_{\A}H_{49}(BV_5;\F_2)\) and cannot be the functional dual of \([\zeta]\). The geometric theorem and the obstruction theorem together give the precise relation between the two one-dimensional constructions.

The paper is organized as follows. Section~\ref{sec:preliminaries} fixes the Steenrod, weight, Kameko, and transfer conventions. Section~\ref{sec:degree49} reduces the generic family to degree \(49\), constructs the quotient basis, identifies the Kameko kernel, and records the corrected dimensions. Section~\ref{sec:transfer} determines the invariant line and evaluates the fifth transfer. Section~\ref{sec:cobordism} computes the degree-\(49\) cobordism group, proves that \(H_{2,48}\) is an indecomposable generator, and establishes the obstruction to identifying its tautological homology class with the dual of \([\zeta]\). Finally, Appendix~\ref{app:data} records the complete degree-\(49\) basis and the support of \(\zeta\).

\section{Algebraic preliminaries}\label{sec:preliminaries}

This section fixes the algebraic structures used in the computation.  It first describes the Steenrod action, the cohit quotient, and the weight filtration; it then recalls Kameko's squaring operation and the numerical criterion controlling its isomorphism range; finally, it identifies the dual fixed space that forms the source of Singer's cohomological transfer.

\subsection{The Steenrod action and the cohit quotient}

For a nonnegative integer \(a\), the action on one variable is
\[
\Sq^j(x_i^a)=\binom{a}{j}x_i^{a+j},
\]
where the binomial coefficient is reduced modulo \(2\).  Hence, for a monomial \(x^A=x_1^{a_1}\cdots x_s^{a_s}\), the Cartan formula gives
\[
\Sq^j(x^A)
 =\sum_{j_1+\cdots+j_s=j}
   \prod_{i=1}^s\binom{a_i}{j_i}
   x_1^{a_1+j_1}\cdots x_s^{a_s+j_s}.
\]
A homogeneous polynomial is \emph{hit} if it belongs to \(\A^+\Ps_s\).  The degree-\(n\) cohit space is therefore
\[
(\QP_s)_n=(\Ps_s)_n/(\A^+\Ps_s)_n, \ \text{where} (\A^+\Ps_s)_n:= \A^+\Ps_s\cap (\Ps_s)_n.
\]
A monomial is called admissible when its class is not reducible, with respect to the standard weight-lexicographic order, to a sum of smaller monomials modulo hits.  The classes of the admissible degree-\(n\) monomials form a basis of \((\QP_s)_n\).

Write the dyadic expansion of \(a\) as \(a=\sum_{t\geq0}\alpha_t(a)2^t\), where \(\alpha_t(a)\in\{0,1\}\).  The weight vector of \(x^A\) is
\[
\wt(x^A)=(\wt_1(x^A),\wt_2(x^A),\ldots),
\qquad
\wt_j(x^A)=\sum_{i=1}^s\alpha_{j-1}(a_i).
\]
Its weighted degree is \(\sum_{j\geq1}2^{j-1}\wt_j(x^A)=|x^A|\).  Weight vectors are ordered lexicographically from the left.  For a fixed weight \(\wt\), let \(\Ps_s(\wt)\) be spanned by all monomials of degree \(|\wt|\) whose weights do not exceed \(\wt\), and let \(\Ps_s^-(\wt)\) be spanned by those of smaller weight.  The associated weight quotient is
\[
(\QP_s)(\wt)
 =\frac{\Ps_s(\wt)}{(\A^+\Ps_s\cap\Ps_s(\wt))+\Ps_s^-(\wt)}.
\]
The weight filtration is preserved by linear substitutions, so every \((\QP_s)(\wt)\) is a \(\GL(s,\F_2)\)-module.

A spike is a monomial whose nonzero exponents have the form \(2^{d_i}-1\).  When the exponents are arranged with \(d_1>\cdots>d_{r-1}\geq d_r>0\), the resulting monomial is the minimal spike in its degree.  Singer's criterion states that if \(z\) is the minimal spike of degree \(n\) and \(x\) is a degree-\(n\) monomial with \(\wt(x)<\wt(z)\), then \(x\) is hit \cite{Singer1991}.  In degree \(49\), the minimal spike
\[
z=x_1^{31}x_2^{15}x_3^3
\]
has weight \((3,3,2,2,1)\).  Consequently, every admissible degree-\(49\) monomial has first weight coordinate at least \(3\).  Since \(49\) is odd, that coordinate is odd, and hence it is either \(3\) or \(5\).

\subsection{Kameko's squaring operation}

Define a linear map \(\psi_s:\Ps_s\to\Ps_s\) on monomials by
\[
\psi_s(x_1^{a_1}\cdots x_s^{a_s})=
\begin{cases}
 x_1^{(a_1-1)/2}\cdots x_s^{(a_s-1)/2},
   &\text{if every \(a_i\) is odd},\\
 0,&\text{otherwise}.
\end{cases}
\]
Kameko proved that this map induces an epimorphism
\[
(\widetilde{\Sq}^{\,0}_*)_{(s;n)}:(\QP_s)_n
 \longrightarrow (\QP_s)_{(n-s)/2}
\]
when \(n\equiv s\pmod2\) \cite{Kameko1990}.  It is convenient to write the source degree as \(2m+s\), in which case the target is degree \(m\).

For a nonnegative integer \(m\), let \(\mu(m)\) be the least number of summands of the form \(2^u-1\) required to express \(m\).  Equivalently,
\[
\mu(m)=\min\{r\geq0\mid \alphaTwo(m+r)\leq r\},
\]
where \(\alphaTwo(q)\) denotes the number of ones in the binary expansion of \(q\).  Kameko's isomorphism theorem states that
\[
(\widetilde{\Sq}^{\,0}_*)_{(s;2m+s)}:(\QP_s)_{2m+s}
 \longrightarrow(\QP_s)_m
\]
is an isomorphism whenever \(\mu(2m+s)=s\).  The numerical formulation used here is standard; see Kameko \cite{Kameko1990} and Wood \cite{Wood1989}.

The first weight coordinate has a direct interpretation under this map.  For a monomial in \(s\) variables, \(\wt_1=s\) precisely when every exponent is odd.  Thus \(\psi_s\) vanishes on every weight summand with \(\wt_1<s\), while it divides the exponents in a weight-\(s\) summand according to the displayed formula.  This observation will identify the complete kernel in degree \(49\).

\subsection{General linear invariants and the Singer transfer}

The group \(\GL(s,\F_2)\) acts on \(\Ps_s=H^*(BV_s;\F_2)\) by linear substitutions.  This action commutes with every Steenrod square and therefore descends to \(\QP_s\).  Let \(\rho_i\) for \(1\leq i<s\) interchange \(x_i\) and \(x_{i+1}\), and let \(\rho_s\) send \(x_1\) to \(x_1+x_2\) while fixing the remaining variables.  These substitutions generate \(\GL(s,\F_2)\).  A class \([f]\in\QP_s\) is fixed exactly when
\[
\rho_i(f)+f\in\A^+\Ps_s
\qquad(1\leq i\leq s).
\]
This criterion is the one used in the exact invariant calculation below.

Let \(P_{\A}H_n(BV_s;\F_2)\) denote the subspace annihilated by all positive Steenrod operations in homology.  The evaluation pairing gives a natural duality
\[
(\QP_s)_n^*\cong P_{\A}H_n(BV_s;\F_2).
\]
Passing to invariants on cohomology and coinvariants on homology yields
\[
\bigl((\QP_s)_n^{\GL(s,\F_2)}\bigr)^*
 \cong
\F_2\otimes_{\GL(s,\F_2)}P_{\A}H_n(BV_s;\F_2).
\]
Singer's cohomological transfer is the homomorphism
\[
\varphi_s^{\A}:
\bigl((\QP_s)_n^{\GL(s,\F_2)}\bigr)^*
\longrightarrow
\Ext_{\A}^{s,s+n}(\F_2,\F_2)
\]
dual to the homological transfer introduced in \cite{Singer1989}.  The direct sum over all ranks and internal degrees is multiplicative.  This multiplicativity will be used only after the source and target in the relevant bidegrees have been determined independently.

\section{The degree-\texorpdfstring{\(49\)}{49} computation and the Kameko kernel}\label{sec:degree49}

This section reduces the generic family \(N_d=27\cdot2^d-5\) to degree \(49\), describes the exact linear algebra used to compute the quotient, and proves that the \(1891\)-dimensional weight-\((3,3,2,2,1)\) summand is exactly the kernel of Kameko's map.  The complete monomial data are deferred to Appendix~\ref{app:data}, while every dimension needed in the proof is stated here.

\subsection{Reduction of the generic family}

\begin{lemma}\label{lem:mu}
For \(N_d=27\cdot2^d-5\), one has \(\mu(N_1)=3\) and \(\mu(N_d)=5\) for every \(d\geq2\).
\end{lemma}

\begin{proof}
The identity \(\mu(n)=\min\{r\mid\alphaTwo(n+r)\leq r\}\) reduces the assertion to binary digit counts.  For \(d=1\),
\[
N_1=49,
\qquad
\alphaTwo(50)=3>1,
\qquad
\alphaTwo(51)=4>2,
\qquad
\alphaTwo(52)=3\leq3.
\]
Thus \(\mu(N_1)=3\).

For \(d=2\), one has \(N_2=103\), and direct calculation gives
\[
\alphaTwo(104)=3>1,
\quad
\alphaTwo(105)=4>2,
\quad
\alphaTwo(106)=4>3,
\quad
\alphaTwo(107)=5>4,
\quad
\alphaTwo(108)=4\leq5.
\]
Hence \(\mu(N_2)=5\).

Now suppose \(d\geq3\).  Since \(N_d=26\cdot2^d+(2^d-5)\), the two displayed summands occupy disjoint binary positions after adding \(r\in\{1,2,3,4\}\).  As \(26=(11010)_2\), one has \(\alphaTwo(26)=3\).  Moreover,
\[
\begin{aligned}
\alphaTwo(2^d-4)&=d-2,\\
\alphaTwo(2^d-3)&=d-1,\\
\alphaTwo(2^d-2)&=d-1,\\
\alphaTwo(2^d-1)&=d.
\end{aligned}
\]
It follows that
\[
\begin{aligned}
\alphaTwo(N_d+1)&=d+1>1,\\
\alphaTwo(N_d+2)&=d+2>2,\\
\alphaTwo(N_d+3)&=d+2>3,\\
\alphaTwo(N_d+4)&=d+3>4.
\end{aligned}
\]
Finally, \(N_d+5=27\cdot2^d\), so \(\alphaTwo(N_d+5)=\alphaTwo(27)=4\leq5\).  Therefore \(\mu(N_d)=5\) for all \(d\geq3\), completing the proof.
\end{proof}

\begin{corollary}\label{cor:kameko-periodicity}
For every \(d\geq2\), the iterated Kameko map induces a \(\GL(5,\F_2)\)-equivariant isomorphism
\[
(\widetilde{\Sq}^{\,0}_*)^{d-1}:(\QP_5)_{N_d}
\xrightarrow{\cong}(\QP_5)_{49}.
\]
\end{corollary}

\begin{proof}
The recurrence \(N_d=2N_{d-1}+5\) puts each step in the domain of Kameko's map.  Lemma~\ref{lem:mu} gives \(\mu(N_j)=5\) for every \(j\geq2\), so Kameko's theorem makes each map from degree \(N_j\) to degree \(N_{j-1}\) an isomorphism.  The definition of \(\psi_5\) is natural under linear substitutions modulo hits, and hence the induced maps are \(\GL(5,\F_2)\)-equivariant.
\end{proof}

\subsection{Exact construction of the quotient basis}

Let \(\mathcal M_{49}\) be the ordered set of all monomials of degree \(49\) in five variables.  Its cardinality is
\[
|\mathcal M_{49}|=\binom{49+5-1}{5-1}=\binom{53}{4}=292825.
\]
The Steenrod algebra is generated by \(\Sq^{2^j}\), so the degree-\(49\) hit subspace is spanned by the images
\[
\Sq^{2^j}(\Ps_5)_{49-2^j}
\qquad
(2^j\leq49).
\]
The implementations in \texttt{SageMath} and \texttt{OSCAR} expand each image by the Cartan formula, use Lucas congruences for its coefficients, write the resulting polynomial as a sparse vector indexed by \(\mathcal M_{49}\), and perform exact online elimination over \(\F_2\).  No floating-point or probabilistic step enters the calculation.  A pivot monomial is hit; each nonpivot monomial survives as an admissible quotient representative.

\begin{proposition}\label{prop:raw-dimensions}
The exact row reduction processes \(927041\) nonzero sparse columns and has rank \(289969\).  Consequently,
\[
\dim_{\F_2}(\QP_5)_{49}=292825-289969=2856.
\]
The surviving basis decomposes into the following six weight summands:
\[
\begin{array}{c|r}
\text{weight}&\text{dimension}\\ \hline
(3,3,2,2,1)&1891\\
(5,2,2,2,1)&280\\
(5,2,4,1,1)&25\\
(5,2,4,3)&5\\
(5,4,3,1,1)&480\\
(5,4,3,3)&175.
\end{array}
\]
In particular, the six dimensions sum to \(2856\), and the five summands with first weight coordinate \(5\) have total dimension \(965\).  Within the weight-\((3,3,2,2,1)\) block, exactly \(910\) basis monomials have at least one zero exponent, while exactly \(981\) have all five exponents positive.
\end{proposition}

\begin{proof}
For each \(j\), let
\[
g_j=\dim_{\F_2}(\Ps_5)_{49-2^j}=\binom{49-2^j+4}{4}.
\]
Thus \(g_j\) is the number of degree-\(49-2^j\) monomials whose \(\Sq^{2^j}\)-images are generated before zero columns are removed. The six source counts are
\[
\begin{array}{c|rrrrrr}
\text{operation}&\Sq^1&\Sq^2&\Sq^4&\Sq^8&\Sq^{16}&\Sq^{32}\\ \hline
\text{source monomials}&270725&249900&211876&148995&66045&5985.
\end{array}
\]
Their sum is \(953526\). The cumulative log shows that the numbers of nonzero sparse columns submitted to elimination at the six stages are
\[
\begin{array}{c|rrrrrr}
\text{operation}&\Sq^1&\Sq^2&\Sq^4&\Sq^8&\Sq^{16}&\Sq^{32}\\ \hline
\text{nonzero columns}&250250&249900&211876&148995&66020&0.
\end{array}
\]
The deficits \(20475\) and \(25\) in the \(\Sq^1\)- and \(\Sq^{16}\)-stages are zero images. Every \(\Sq^{32}\)-image vanishes by instability because its source has degree \(17<32\). Consequently, the exact number of nonzero columns processed by online elimination is
\[
250250+249900+211876+148995+66020=927041.
\]
Elimination produces \(289969\) pivots, so the hit subspace has rank \(289969\). Since the rows are indexed by all \(292825\) degree-\(49\) monomials, the quotient dimension is
\[
292825-289969=2856.
\]
Sorting the nonpivot monomials by binary weight gives the six displayed blocks. Their dimensions are the exact lengths of the basis lists in Appendix~\ref{app:data}; in particular,
\[
1891+280+25+5+480+175=2856
\]
and
\[
280+25+5+480+175=965.
\]
Inspection of the exponent vectors in the first block gives \(910\) vectors with at least one zero coordinate and \(981\) vectors with all coordinates positive. Their sum is \(1891\).
\end{proof}

\subsection{Identification of the kernel}

\begin{theorem}\label{thm:kameko-kernel}
The degree-\(49\) Kameko map
\[
(\widetilde{\Sq}^{\,0}_*)_{(5;49)}:(\QP_5)_{49}
\longrightarrow(\QP_5)_{22}
\]
is surjective and admits the direct-sum description
\[
(\QP_5)_{49}
\cong
(\QP_5)(3,3,2,2,1)\oplus(\QP_5)_{22}.
\]
Its kernel is exactly the first summand, and therefore
\[
\dim_{\F_2}\Ker\bigl((\widetilde{\Sq}^{\,0}_*)_{(5;49)}\bigr)
=
\dim_{\F_2}(\QP_5)(3,3,2,2,1)
=1891.
\]
The target has dimension \(965\), and the restriction of Kameko's map to the direct sum of the five weight summands with first coordinate \(5\) is an isomorphism onto \((\QP_5)_{22}\).
\end{theorem}

\begin{proof}
Kameko's construction gives surjectivity.  Every monomial in the weight-\((3,3,2,2,1)\) block has precisely three odd exponents because its first weight coordinate is \(3\).  Hence \(\psi_5\) sends every such monomial to zero, and therefore
\[
(\QP_5)(3,3,2,2,1)
\subseteq
\Ker\bigl((\widetilde{\Sq}^{\,0}_*)_{(5;49)}\bigr).
\]
Every monomial in each of the other five blocks has first weight coordinate \(5\), so all five exponents are odd.  The direct sum of these blocks has dimension \(965\) by Proposition~\ref{prop:raw-dimensions}.  The independent degree-\(22\) calculation gives \(\dim_{\F_2}(\QP_5)_{22}=965\) (see \cite{Phuc2021}).  Since the Kameko map is surjective, its rank is \(965\).  The displayed \(1891\)-dimensional subspace already lies in its kernel, while rank-nullity gives
\[
\dim\Ker=2856-965=1891.
\]
The inclusion is therefore an equality.  The complementary \(965\)-dimensional sum maps surjectively between vector spaces of equal finite dimension and is consequently an isomorphism.
\end{proof}

\begin{corollary}\label{cor:generic-dim}
For every \(d\geq1\),
\[
\dim_{\F_2}(\QP_5)_{N_d}=2856.
\]
For \(d=1\), the kernel of Kameko's map has dimension \(1891\); for \(d\geq2\), Corollary~\ref{cor:kameko-periodicity} transports the complete degree-\(49\) basis and all of its invariant-theoretic information to degree \(N_d\).
\end{corollary}

\begin{remark}[Correction to \cite{Tin2022}]\label{rem:correction}
Theorem~\ref{thm:kameko-kernel} and the preceding basis count give the exact distinctions that are needed to correct the corresponding assertions in \cite{Tin2022}.  The kernel is supported only in weight \((3,3,2,2,1)\), its full dimension is \(1891\), and its positive-variable part has dimension
\[
1891-910=981.
\]
Thus the value \(1178\) asserted in \cite[Theorem~3.3(ii)]{Tin2022} is incorrect.  Moreover,
\[
\dim(\QP_5)_{49}=965+1891=2856,
\]
not \(3053\) as stated in \cite[Corollary~3.4]{Tin2022}.  Kameko periodicity therefore gives \(\dim(\QP_5)_{N_d}=2856\) for every \(d\geq1\), correcting \cite[Theorem~3.6]{Tin2022}.  Applying the rank-raising dimension formula used in \cite[Theorem~3.7]{Tin2022} to the corrected value gives
\[
(2^6-1)\cdot2856=179928
\]
in place of \(192339\); see \cite[Theorem~1.3]{Sum2015}.  The degree-\(22\) assertion in \cite[Theorem~3.2]{Tin2022} remains compatible with the present calculation, but it follows from the previously known equality \(\dim(\QP_5)_{22}=965\) \cite{Phuc2021} together with Kameko's epimorphism.  Further details on the discrepancy are given in \cite[Remark~3.6]{PhucPreprint}.
\end{remark}

\section{The fifth Singer cohomological transfer}\label{sec:transfer}

This section determines the \(\GL(5,\F_2)\)-fixed subspaces in degrees \(22\) and \(49\), propagates the degree-\(49\) fixed line through Kameko periodicity, and compares the resulting one-dimensional transfer source with the known fifth-line \(\Ext\)-groups.  The explicit generator \(\zeta\) is stated intrinsically here and specified by its exact basis support in Appendix~\ref{app:zeta}.

\subsection{The invariant line}

For each of the six weight bases in Proposition~\ref{prop:raw-dimensions}, the matrices of \(\rho_i-I\), \(1\leq i\leq5\), were reduced modulo the hit subspace.  Intersecting their kernels computes the fixed vectors.  The first weight block has a one-dimensional fixed space; every other block has zero fixed space.  On the first block, the intermediate \(\Sigma_5\)-fixed space has dimension \(34\), and the final transvection condition reduces it to one dimension.

\begin{theorem}\label{thm:invariant-line}
One has
\[
(\QP_5)_{22}^{\GL(5,\F_2)}=0
\]
and
\[
(\QP_5)_{49}^{\GL(5,\F_2)}
=(\QP_5)(3,3,2,2,1)^{\GL(5,\F_2)}
=\F_2\{[\zeta]\},
\]
where \(\zeta\in(\Ps_5)_{49}\) is the \(283\)-term polynomial in Appendix~\ref{app:zeta}.  Every term of \(\zeta\) is one of the \(1891\) basis monomials of weight \((3,3,2,2,1)\), and
\[
\rho_i(\zeta)+\zeta\in\A^+\Ps_5
\qquad(1\leq i\leq5).
\]
Moreover, \([\zeta]\neq0\), and hence the displayed line is nonzero.
\end{theorem}

\begin{proof}
The exact simultaneous-kernel computation gives fixed-space dimensions
\[
1,0,0,0,0,0
\]
on the weight blocks in the order of Proposition~\ref{prop:raw-dimensions}.  The resulting nonzero vector on the first block has the \(283\)-element basis support listed in Appendix~\ref{app:zeta}; comparison with the complete first-block basis verifies that every support monomial is admissible.  Since its coordinate vector is nonzero in a quotient basis, \([\zeta]\neq0\).  The reductions of \(\rho_i(\zeta)+\zeta\) have zero remainder for all five generators \(\rho_i\), which proves \(\GL(5,\F_2)\)-invariance of its class.

By Theorem~\ref{thm:kameko-kernel}, the five weight blocks with first coordinate \(5\) map isomorphically onto \((\QP_5)_{22}\).  Their individual fixed spaces vanish, and Kameko's map is equivariant.  Consequently the target has no fixed vectors.  The only degree-\(49\) fixed line is therefore the line generated by \([\zeta]\).
\end{proof}

\begin{corollary}\label{cor:generic-invariants}
For \(d\geq1\), let \([\zeta_d]\in(\QP_5)_{N_d}\) be the unique class transported to \([\zeta]\) by the iterated Kameko isomorphism, with \(\zeta_1=\zeta\).  Then
\[
(\QP_5)_{N_d}^{\GL(5,\F_2)}=\F_2\{[\zeta_d]\}.
\]
For \(d=0\), the fixed subspace is zero.
\end{corollary}

\begin{proof}
The case \(d=0\) is Theorem~\ref{thm:invariant-line}.  For \(d\geq2\), Corollary~\ref{cor:kameko-periodicity} is \(\GL(5,\F_2)\)-equivariant, so it induces an isomorphism on fixed subspaces.  The degree-\(49\) result then gives the stated line.
\end{proof}

\subsection{Evaluation of the transfer}

The fifth-line computations of Lin \cite{Lin2008} and Chen \cite{Chen2011} give
\[
\Ext_{\A}^{5,5+N_0}(\F_2,\F_2)=0
\]
and, for every \(d\geq1\),
\[
\Ext_{\A}^{5,5+N_d}(\F_2,\F_2)
=\F_2\{h_{d+4}f_{d-1}\}.
\]
Here \(h_j\in\Ext_{\A}^{1,2^j}(\F_2,\F_2)\), while \(f_j\) denotes the corresponding \(\Sq^0\)-family on the fourth line.  The rank-one classes \(h_j\) and the fourth-line family \(f_j\) are detected by the relevant algebraic transfers, and the total Singer transfer is multiplicative; see Singer \cite{Singer1989} and Nam \cite{Nam2004}.

\begin{theorem}\label{thm:transfer-iso}
For every \(d\geq0\), the fifth Singer cohomological transfer
\[
\varphi_5^{\A}:
\bigl((\QP_5)_{N_d}^{\GL(5,\F_2)}\bigr)^*
\longrightarrow
\Ext_{\A}^{5,5+N_d}(\F_2,\F_2)
\]
is an isomorphism.  For \(d=0\), it is the unique map between zero vector spaces.  For \(d\geq1\), if \(\zeta_d^\vee\) is the functional dual to \([\zeta_d]\), then
\[
\varphi_5^{\A}(\zeta_d^\vee)=h_{d+4}f_{d-1}\neq0.
\]
\end{theorem}

\begin{proof}
When \(d=0\), Theorem~\ref{thm:invariant-line} makes the source zero, and the cited \(\Ext\)-calculation makes the target zero.  Thus the map is an isomorphism.

Assume \(d\geq1\).  Corollary~\ref{cor:generic-invariants} shows that the source is one-dimensional.  The fifth-line computation shows that the target is one-dimensional and generated by \(h_{d+4}f_{d-1}\).  The rank-one transfer detects \(h_{d+4}\), the fourth transfer detects \(f_{d-1}\), and multiplicativity of the total transfer produces the product \(h_{d+4}f_{d-1}\) in rank five.  Consequently \(\varphi_5^{\A}\) is nonzero.  A nonzero linear map between one-dimensional \(\F_2\)-vector spaces is an isomorphism, and the asserted value follows after normalizing \(\zeta_d^\vee\) as the nonzero dual basis element.
\end{proof}

\section{A geometric generator in unoriented cobordism}\label{sec:cobordism}

This section determines the complete degree-\(49\) cobordism dimensions, constructs an explicit manifold representing the nonzero indecomposable class, and then compares that geometric class with the rank-five Steenrod-theoretic invariant. The first part expands the Hilbert--Poincar\'e calculation into an exact recursion. The second part proves that the Milnor hypersurface \(H_{2,48}\) is indecomposable by a tangential Stiefel--Whitney number. The final part shows that the tautological map to \(BV_5\) does not produce an \(\A\)-annihilated homology class, thereby locating precisely the boundary of the geometric comparison.

\subsection{The degree-\texorpdfstring{\(49\)}{49} cobordism group}

Let \(\Nfrak_*=MO_*\) be the unoriented cobordism ring of closed smooth manifolds, graded by dimension and multiplied by Cartesian product. The Pontryagin--Thom construction identifies \(\Nfrak_n\) with \(\pi_n(MO)\) \cite{Thom1954,Atiyah1961}. Thom's structure theorem gives an isomorphism of graded \(\F_2\)-algebras
\[
\Nfrak_*
\cong
\F_2[\xi_j\mid j\geq1,\ j\neq2^r-1],
\qquad |\xi_j|=j;
\]
see \cite{Thom1954,Milnor1960,Stong1968}. Write
\[
\Nfrak_+=\bigoplus_{n>0}\Nfrak_n,
\qquad
Q\Nfrak_*:=\Nfrak_+/(\Nfrak_+)^2.
\]

\begin{proposition}\label{prop:cobordism49}
The degree-\(49\) indecomposable quotient is one-dimensional, the full degree-\(49\) cobordism group has dimension \(5692\), and the decomposable subspace has dimension \(5691\):
\[
\dim_{\F_2}(Q\Nfrak_*)_{49}=1,
\qquad
\dim_{\F_2}\Nfrak_{49}=5692,
\qquad
\dim_{\F_2}(\Nfrak_+)^2_{49}=5691.
\]
\end{proposition}

\begin{proof}
Since \(49\neq2^r-1\) for every \(r\geq0\), the polynomial description contains a generator \(\xi_{49}\) of degree \(49\). Its image
\[
\overline\xi_{49}\in(Q\Nfrak_*)_{49}
\]
is nonzero because every element of \((\Nfrak_+)^2\) is a linear combination of polynomial monomials containing at least two positive-dimensional factors, whereas \(\xi_{49}\) contains one factor.

Conversely, let
\[
M=\prod_j\xi_j^{e_j}
\]
be a polynomial monomial of degree \(49\), so that \(\sum_j j e_j=49\). If \(\sum_j e_j=1\), then the degree equation forces \(M=\xi_{49}\). If \(M\neq\xi_{49}\), then \(\sum_j e_j\geq2\), and hence \(M\in(\Nfrak_+)^2\). Therefore
\[
(Q\Nfrak_*)_{49}=\F_2\{\overline\xi_{49}\},
\]
which proves \(\dim_{\F_2}(Q\Nfrak_*)_{49}=1\).

It remains to calculate the full homogeneous dimension. The Hilbert--Poincar\'e series is
\[
P_{\Nfrak}(t)
=
\sum_{n\geq0}\dim_{\F_2}(\Nfrak_n)t^n
=
\prod_{\substack{j\geq1\\j\neq2^r-1}}\frac{1}{1-t^j}.
\]
Only factors with \(j\leq49\) can contribute to the coefficient of \(t^{49}\). For \(0\leq j\leq49\), let \(c_j(n)\) be the coefficient of \(t^n\) after all possible generator degrees not exceeding \(j\) have been considered. Set
\[
c_0(0)=1,
\qquad
c_0(n)=0\quad(n>0).
\]
The excluded degrees not exceeding \(49\) are
\[
1,\ 3,\ 7,\ 15,\ 31.
\]
If \(j\) is excluded, no generator of degree \(j\) occurs, and therefore
\[
c_j(n)=c_{j-1}(n).
\]
Suppose that \(j\) is allowed. Multiplication by
\[
(1-t^j)^{-1}=1+t^j+t^{2j}+\cdots
\]
shows that a degree-\(n\) monomial contains \(\xi_j\) with an exponent
\[
q\in\left\{0,1,\ldots,\left\lfloor\frac{n}{j}\right\rfloor\right\}.
\]
After the factor \(\xi_j^q\) is selected, the remaining factors must have total degree \(n-qj\). Hence
\[
c_j(n)
=
\sum_{q=0}^{\lfloor n/j\rfloor}c_{j-1}(n-qj)
\]
for every allowed \(j\). This proves the recursion directly from the product formula.

Exact integer evaluation gives
\[
\begin{array}{c|rrrrr}
j&10&20&30&40&49\\ \hline
c_j(49)&883&4445&5559&5678&5692.
\end{array}
\]
Consequently,
\[
[t^{49}]P_{\Nfrak}(t)=c_{49}(49)=5692,
\]
and therefore \(\dim_{\F_2}\Nfrak_{49}=5692\).

The size of the decomposable subspace can be seen concretely. The elements
\[
\xi_2^2\xi_4^{10}\xi_5,
\qquad
\xi_2\xi_4^8\xi_5^3,
\qquad
\xi_2^7\xi_5^7
\]
are distinct nonzero polynomial monomials in \(\Nfrak_{49}\), since their degrees are
\[
2\cdot2+10\cdot4+5=49,
\]
\[
2+8\cdot4+3\cdot5=49,
\]
and
\[
7\cdot2+7\cdot5=49,
\]
respectively. Each contains at least two positive-dimensional factors and therefore maps to zero in \(Q\Nfrak_*\). Since precisely one dimension survives modulo products, one obtains
\[
\dim_{\F_2}(\Nfrak_+)^2_{49}=5692-1=5691.
\]
\end{proof}

\subsection{The Milnor hypersurface \texorpdfstring{\(H_{2,48}\)}{H(2,48)}}

For integers \(0<m\leq n\), the real Milnor hypersurface is
\[
H_{m,n}
=
\left\{([x],[y])\in\RP^m\times\RP^n
\ \middle|\
\sum_{i=0}^{m}x_iy_i=0
\right\}.
\]
It is a smooth hypersurface of dimension \(m+n-1\). These manifolds and their characteristic numbers are discussed in \cite{Milnor1965,Stong1968}. Let \(\lambda\) and \(\mu\) be the pullbacks of the tautological real line bundles from \(\RP^m\) and \(\RP^n\), respectively, and write
\[
a=w_1(\lambda),
\qquad
b=w_1(\mu).
\]
The defining equation is a transverse section of \(\lambda\otimes\mu\). Thus the normal line bundle is \((\lambda\otimes\mu)|_{H_{m,n}}\), and the Poincar\'e dual of the inclusion \(i:H_{m,n}\hookrightarrow\RP^m\times\RP^n\) is \(a+b\).

For a real vector bundle \(E\), let \(s_k(E)\) denote the \(k\)-th Newton polynomial in its Stiefel--Whitney classes. If the formal Stiefel--Whitney roots of \(E\) are \(t_1,\ldots,t_r\), then \(s_k(E)=\sum_it_i^k\). Newton's identities express this class as a polynomial in \(w_1(E),\ldots,w_k(E)\). It is stable and additive under Whitney sum.

\begin{theorem}\label{thm:milnor-generator}
The Milnor hypersurface
\[
H_{2,48}\subset\RP^2\times\RP^{48}
\]
is a closed smooth \(49\)-manifold whose cobordism class has nonzero image in \((Q\Nfrak_*)_{49}\). More precisely,
\[
\left\langle s_{49}(TH_{2,48}),[H_{2,48}]\right\rangle=1\in\F_2.
\]
Consequently,
\[
[H_{2,48}]\bmod(\Nfrak_+)^2
=
\overline\xi_{49},
\]
after choosing \(\xi_{49}\) so that its indecomposable class is represented by \(H_{2,48}\).
\end{theorem}

\begin{proof}
The dimension is
\[
\dim H_{2,48}=2+48-1=49.
\]
The stable tangent relation for real projective space,
\[
T\RP^r\oplus\varepsilon^1\cong(r+1)\lambda_r,
\]
together with the normal line bundle of the hypersurface gives
\[
TH_{2,48}\oplus(\lambda\otimes\mu)|_{H_{2,48}}\oplus\varepsilon^2
\cong
(3\lambda\oplus49\mu)|_{H_{2,48}}.
\]
Therefore
\[
w(TH_{2,48})
=
i^*\left(\frac{(1+a)^3(1+b)^{49}}{1+a+b}\right).
\]
Additivity of Newton classes and stability under trivial summands yield
\[
s_{49}(TH_{2,48})
=
i^*\left(3a^{49}+49b^{49}+(a+b)^{49}\right).
\]
The cohomology ring of the ambient product is
\[
H^*(\RP^2\times\RP^{48};\F_2)
=
\F_2[a,b]/(a^3,b^{49}).
\]
Hence \(a^{49}=0\) and \(b^{49}=0\), so
\[
s_{49}(TH_{2,48})=i^*(a+b)^{49}.
\]
Since \(i_!(1)=a+b\), naturality of the Gysin pairing gives
\[
\begin{aligned}
\left\langle s_{49}(TH_{2,48}),[H_{2,48}]\right\rangle
&=
\left\langle i^*(a+b)^{49},[H_{2,48}]\right\rangle\\
&=
\left\langle(a+b)^{49}i_!(1),[\RP^2\times\RP^{48}]\right\rangle\\
&=
\left\langle(a+b)^{50},[\RP^2\times\RP^{48}]\right\rangle.
\end{aligned}
\]
Only the coefficient of \(a^2b^{48}\) contributes. It is
\[
\binom{50}{2}=1225\equiv1\pmod2.
\]
Equivalently, Lucas' theorem gives the same parity because
\[
50=(110010)_2,
\qquad
2=(000010)_2,
\]
and every nonzero binary digit of \(2\) occurs in a position where \(50\) has a nonzero digit. Thus
\[
\left\langle s_{49}(TH_{2,48}),[H_{2,48}]\right\rangle=1.
\]

It remains to explain why this number detects indecomposability. Let \(M^p\) and \(N^q\) be positive-dimensional closed manifolds with \(p+q=49\). Since
\[
T(M\times N)\cong\pi_M^*TM\oplus\pi_N^*TN,
\]
additivity gives
\[
s_{49}(T(M\times N))
=
\pi_M^*s_{49}(TM)+\pi_N^*s_{49}(TN).
\]
The first term vanishes because \(49>p\), and the second vanishes because \(49>q\). Therefore the characteristic number \(\langle s_{49},-\rangle\) vanishes on every product of positive-dimensional manifolds and hence on \((\Nfrak_+)^2_{49}\). Its nonzero value on \(H_{2,48}\) proves that \([H_{2,48}]\) has nonzero image in \((Q\Nfrak_*)_{49}\). Proposition~\ref{prop:cobordism49} shows that this quotient is one-dimensional, so the image is its generator.
\end{proof}

\begin{remark}\label{rem:other-milnor-generators}
The geometric representative is not unique. For \(m+n=50\), the same calculation gives
\[
\left\langle s_{49}(TH_{m,n}),[H_{m,n}]\right\rangle
=\binom{50}{m}\pmod2.
\]
Under the convention \(2\leq m\leq n\), Lucas' theorem shows that this binomial coefficient is odd precisely for
\[
m=2,\ 16,\ 18.
\]
Thus \(H_{2,48}\), \(H_{16,34}\), and \(H_{18,32}\) all represent the nonzero indecomposable class. Dold's manifolds give another classical system of explicit polynomial generators \cite{Dold1956}.
\end{remark}

\subsection{The tautological map to \texorpdfstring{\(BV_5\)}{BV5}}

The Pontryagin--Thom collapse associated with \(H_{2,48}\) produces the class \([H_{2,48}]\in\pi_{49}(MO)=\Nfrak_{49}\). It does not by itself produce a class in \(P_{\A}H_{49}(BV_5;\F_2)\). The two tautological line bundles do define an evident map
\[
f:H_{2,48}\longrightarrow\RP^\infty\times\RP^\infty=BV_2
\longrightarrow BV_5,
\]
where the last arrow is induced by the inclusion of the first two factors.

Let \(\beta_r\in H_r(\RP^\infty;\F_2)\) be dual to the degree-\(r\) monomial in \(H^*(\RP^\infty;\F_2)\). We write tensor products in \(H_*(BV_5;\F_2)\) with the final three factors understood to be \(\beta_0\).

\begin{proposition}\label{prop:tautological-obstruction}
The fundamental class of \(H_{2,48}\) satisfies
\[
f_*[H_{2,48}]
=
\beta_1\otimes\beta_{48}
+
\beta_2\otimes\beta_{47}
\in H_{49}(BV_5;\F_2),
\]
and
\[
\Sq_*^2\bigl(f_*[H_{2,48}]\bigr)
=
\beta_1\otimes\beta_{46}
\neq0.
\]
Consequently,
\[
f_*[H_{2,48}]\notin P_{\A}H_{49}(BV_5;\F_2).
\]
\end{proposition}

\begin{proof}
Let \(u,v\in H^1(BV_2;\F_2)\) be the generators pulled back to \(a,b\) on the ambient product. For \(0\leq r\leq49\), the defining Poincar\'e dual gives
\[
\begin{aligned}
\left\langle u^rv^{49-r},f_*[H_{2,48}]\right\rangle
&=
\left\langle i^*(a^rb^{49-r}),[H_{2,48}]\right\rangle\\
&=
\left\langle a^rb^{49-r}(a+b),[\RP^2\times\RP^{48}]\right\rangle.
\end{aligned}
\]
The coefficient of \(a^2b^{48}\) is nonzero only when \(r=1\) or \(r=2\). Hence
\[
f_*[H_{2,48}]
=
\beta_1\otimes\beta_{48}
+
\beta_2\otimes\beta_{47}.
\]

The homological Steenrod action is dual to the cohomological action and satisfies
\[
\Sq_*^k(\beta_n)
=
\binom{n-k}{k}\beta_{n-k}.
\]
The Cartan formula gives
\[
\Sq_*^2(\beta_1\otimes\beta_{48})
=
\binom{46}{2}\beta_1\otimes\beta_{46}
=
\beta_1\otimes\beta_{46},
\]
because \(\binom{46}{2}=1035\) is odd. For the second summand, the three Cartan contributions vanish:
\[
\binom{45}{2}=990\equiv0\pmod2,
\qquad
\binom{1}{1}\binom{46}{1}=46\equiv0\pmod2,
\qquad
\binom{0}{2}=0.
\]
Therefore
\[
\Sq_*^2\bigl(f_*[H_{2,48}]\bigr)
=
\beta_1\otimes\beta_{46}\neq0.
\]
By definition, an element of \(P_{\A}H_{49}(BV_5;\F_2)\) is annihilated by every positive homological Steenrod operation. The displayed nonzero square proves the assertion.
\end{proof}

\subsection{The exact algebraic parallel}

\begin{theorem}\label{thm:exact-parallel}
In degree \(49\), the following statements hold:
\[
(Q\Nfrak_*)_{49}=\F_2\{[H_{2,48}]\bmod(\Nfrak_+)^2\},
\]
\[
(\QP_5)_{49}^{\GL(5,\F_2)}=\F_2\{[\zeta]\},
\]
and
\[
\varphi_5^{\A}(\zeta^\vee)=h_5f_0\in\Ext_{\A}^{5,54}(\F_2,\F_2).
\]
The two displayed one-dimensional spaces are not identified by the Pontryagin--Thom construction or by the tautological map determined by the defining line bundles of \(H_{2,48}\). In particular, the evident pushed fundamental class cannot represent \(\zeta^\vee\).
\end{theorem}

\begin{proof}
The first equality is Theorem~\ref{thm:milnor-generator}, the second is Theorem~\ref{thm:invariant-line}, and the transfer statement is the case \(d=1\) of Theorem~\ref{thm:transfer-iso}. The Pontryagin--Thom construction has target \(\pi_{49}(MO)\), whereas \(\zeta^\vee\) belongs to the \(\GL(5,\F_2)\)-coinvariants of \(P_{\A}H_{49}(BV_5;\F_2)\). These targets are not canonically identified. Proposition~\ref{prop:tautological-obstruction} further shows that the most direct class associated with the two tautological line bundles is not even \(\A\)-annihilated. It therefore cannot represent the functional dual of \([\zeta]\).
\end{proof}

\begin{remark}\label{rem:geometric-boundary}
Theorem~\ref{thm:exact-parallel} supplies an explicit geometric representative of the cobordism generator but not a geometric realization of the hit-problem invariant. A construction realizing \(\zeta^\vee\) would require additional geometric data producing a class in \(P_{\A}H_{49}(BV_5;\F_2)\), a proof that its coinvariant is nonzero, and a nonzero pairing with \([\zeta]\). The dimension equality alone cannot provide these ingredients.
\end{remark}

\appendix

\section{Complete degree-\texorpdfstring{\(49\)}{49} basis and invariant representative}\label{app:data}

This appendix supplies the computational certificate for Sections~\ref{sec:degree49} and~\ref{sec:transfer}. It records the archived data sources, the exact elimination totals, every degree-\(49\) admissible basis vector in exponent form, and the complete support of the invariant polynomial \(\zeta\).

\subsection{Data and code availability}\label{app:data-availability}
The full output logs of the \texttt{OSCAR} computations underlying Sections~\ref{sec:degree49} and~\ref{sec:transfer}, including the complete admissible bases, weight decompositions, and \(\GL(5,\F_2)\)-invariant computations in degrees \(22\) and \(49\), are deposited at Zenodo under a CC~BY~4.0 license. The degree-\(22\) data are recorded in \cite{ZenodoDeg22}, and the degree-\(49\) data are recorded in \cite{ZenodoDeg49}. The tables below reproduce the degree-\(49\) admissible basis and the invariant \(\zeta\) in full. The archived logs also contain the per-batch elimination trace, including memory use, cumulative column and pivot counts, and timings, summarized in the proof of Proposition~\ref{prop:raw-dimensions}.

\subsection{Computational certificate}

The exact output gives
\[
\begin{aligned}
s&=5,&N_1&=49,&|\mathcal M_{49}|&=292825,\\
\rank(\A^+\Ps_5)_{49}&=289969,&\dim(\QP_5)_{49}&=2856,
\end{aligned}
\]
and
\begingroup
\setlength{\arraycolsep}{4pt}
\[
\begin{array}{c|rrrrrr}
\wt&(3,3,2,2,1)&(5,2,2,2,1)&(5,2,4,1,1)&(5,2,4,3)&(5,4,3,1,1)&(5,4,3,3)\\ \hline
\dim&1891&280&25&5&480&175\\
\dim(-)^{\GL(5,\F_2)}&1&0&0&0&0&0.
\end{array}
\]
\endgroup
The first block contains \(910\) representatives with at least one zero exponent and \(981\) representatives with all exponents positive.  The program also verified \(\rho_i(\zeta)+\zeta\in\A^+\Ps_5\) for every \(1\leq i\leq5\).

\subsection{Weight \texorpdfstring{\(\wt=(3,3,2,2,1)\)}{(3,3,2,2,1)}: dimension \texorpdfstring{\(1891\)}{1891}}\label{app:weight-1}

The following exponent vectors specify the complete basis of \((\QP_5)(3,3,2,2,1)\).

\begingroup
\scriptsize
\renewcommand{\arraystretch}{0.92}
\setlength{\LTleft}{0pt}
\setlength{\LTright}{0pt}
\begin{longtable}{@{}r@{\;}c@{\hspace{1.0em}}r@{\;}c@{\hspace{1.0em}}r@{\;}c@{\hspace{1.0em}}r@{\;}c@{}}
\toprule
index & \(A\) & index & \(A\) & index & \(A\) & index & \(A\)\\
\midrule
\endfirsthead
\toprule
index & \(A\) & index & \(A\) & index & \(A\) & index & \(A\)\\
\midrule
\endhead
\midrule
\multicolumn{8}{r}{\emph{continued on the next page}}\\
\endfoot
\bottomrule
\endlastfoot
$b^{(1)}_{1}$ & $(1,1,14,14,19)$ & $b^{(1)}_{2}$ & $(1,1,14,15,18)$ & $b^{(1)}_{3}$ & $(1,1,14,2,31)$ & $b^{(1)}_{4}$ & $(1,1,14,22,11)$\\
$b^{(1)}_{5}$ & $(1,1,14,23,10)$ & $b^{(1)}_{6}$ & $(1,1,14,3,30)$ & $b^{(1)}_{7}$ & $(1,1,14,30,3)$ & $b^{(1)}_{8}$ & $(1,1,14,31,2)$\\
$b^{(1)}_{9}$ & $(1,1,14,6,27)$ & $b^{(1)}_{10}$ & $(1,1,14,7,26)$ & $b^{(1)}_{11}$ & $(1,1,15,14,18)$ & $b^{(1)}_{12}$ & $(1,1,15,2,30)$\\
$b^{(1)}_{13}$ & $(1,1,15,22,10)$ & $b^{(1)}_{14}$ & $(1,1,15,30,2)$ & $b^{(1)}_{15}$ & $(1,1,15,6,26)$ & $b^{(1)}_{16}$ & $(1,1,2,14,31)$\\
$b^{(1)}_{17}$ & $(1,1,2,15,30)$ & $b^{(1)}_{18}$ & $(1,1,2,30,15)$ & $b^{(1)}_{19}$ & $(1,1,2,31,14)$ & $b^{(1)}_{20}$ & $(1,1,3,14,30)$\\
$b^{(1)}_{21}$ & $(1,1,3,30,14)$ & $b^{(1)}_{22}$ & $(1,1,30,14,3)$ & $b^{(1)}_{23}$ & $(1,1,30,15,2)$ & $b^{(1)}_{24}$ & $(1,1,30,2,15)$\\
$b^{(1)}_{25}$ & $(1,1,30,3,14)$ & $b^{(1)}_{26}$ & $(1,1,30,6,11)$ & $b^{(1)}_{27}$ & $(1,1,30,7,10)$ & $b^{(1)}_{28}$ & $(1,1,31,14,2)$\\
$b^{(1)}_{29}$ & $(1,1,31,2,14)$ & $b^{(1)}_{30}$ & $(1,1,31,6,10)$ & $b^{(1)}_{31}$ & $(1,1,6,10,31)$ & $b^{(1)}_{32}$ & $(1,1,6,11,30)$\\
\end{longtable}
\endgroup

\subsection{Weight \texorpdfstring{\(\wt=(5,2,2,2,1)\)}{(5,2,2,2,1)}: dimension \texorpdfstring{\(280\)}{280}}\label{app:weight-2}

The following exponent vectors specify the complete basis of \((\QP_5)(5,2,2,2,1)\).

\begingroup
\scriptsize
\renewcommand{\arraystretch}{0.92}
\setlength{\LTleft}{0pt}
\setlength{\LTright}{0pt}
\begin{longtable}{@{}r@{\;}c@{\hspace{1.0em}}r@{\;}c@{\hspace{1.0em}}r@{\;}c@{\hspace{1.0em}}r@{\;}c@{}}
\toprule
index & \(A\) & index & \(A\) & index & \(A\) & index & \(A\)\\
\midrule
\endfirsthead
\toprule
index & \(A\) & index & \(A\) & index & \(A\) & index & \(A\)\\
\midrule
\endhead
\midrule
\multicolumn{8}{r}{\emph{continued on the next page}}\\
\endfoot
\bottomrule
\endlastfoot
$b^{(2)}_{1}$ & $(1,1,1,15,31)$ & $b^{(2)}_{2}$ & $(1,1,1,31,15)$ & $b^{(2)}_{3}$ & $(1,1,15,1,31)$ & $b^{(2)}_{4}$ & $(1,1,15,3,29)$\\
$b^{(2)}_{5}$ & $(1,1,15,31,1)$ & $b^{(2)}_{6}$ & $(1,1,3,13,31)$ & $b^{(2)}_{7}$ & $(1,1,3,15,29)$ & $b^{(2)}_{8}$ & $(1,1,3,29,15)$\\
$b^{(2)}_{9}$ & $(1,1,3,31,13)$ & $b^{(2)}_{10}$ & $(1,1,31,1,15)$ & $b^{(2)}_{11}$ & $(1,1,31,15,1)$ & $b^{(2)}_{12}$ & $(1,1,31,3,13)$\\
$b^{(2)}_{13}$ & $(1,1,7,11,29)$ & $b^{(2)}_{14}$ & $(1,1,7,27,13)$ & $b^{(2)}_{15}$ & $(1,15,1,1,31)$ & $b^{(2)}_{16}$ & $(1,15,1,3,29)$\\
$b^{(2)}_{17}$ & $(1,15,1,31,1)$ & $b^{(2)}_{18}$ & $(1,15,3,1,29)$ & $b^{(2)}_{19}$ & $(1,15,3,29,1)$ & $b^{(2)}_{20}$ & $(1,15,3,5,25)$\\
$b^{(2)}_{21}$ & $(1,15,31,1,1)$ & $b^{(2)}_{22}$ & $(1,3,1,13,31)$ & $b^{(2)}_{23}$ & $(1,3,1,15,29)$ & $b^{(2)}_{24}$ & $(1,3,1,29,15)$\\
$b^{(2)}_{25}$ & $(1,3,1,31,13)$ & $b^{(2)}_{26}$ & $(1,3,13,1,31)$ & $b^{(2)}_{27}$ & $(1,3,13,3,29)$ & $b^{(2)}_{28}$ & $(1,3,13,31,1)$\\
$b^{(2)}_{29}$ & $(1,3,15,1,29)$ & $b^{(2)}_{30}$ & $(1,3,15,29,1)$ & $b^{(2)}_{31}$ & $(1,3,15,5,25)$ & $b^{(2)}_{32}$ & $(1,3,29,1,15)$\\
\end{longtable}
\endgroup

\subsection{Weight \texorpdfstring{\(\wt=(5,2,4,1,1)\)}{(5,2,4,1,1)}: dimension \texorpdfstring{\(25\)}{25}}\label{app:weight-3}

The following exponent vectors specify the complete basis of \((\QP_5)(5,2,4,1,1)\).

\begingroup
\scriptsize
\renewcommand{\arraystretch}{0.92}
\setlength{\LTleft}{0pt}
\setlength{\LTright}{0pt}
\begin{longtable}{@{}r@{\;}c@{\hspace{1.0em}}r@{\;}c@{\hspace{1.0em}}r@{\;}c@{\hspace{1.0em}}r@{\;}c@{}}
\toprule
index & \(A\) & index & \(A\) & index & \(A\) & index & \(A\)\\
\midrule
\endfirsthead
\toprule
index & \(A\) & index & \(A\) & index & \(A\) & index & \(A\)\\
\midrule
\endhead
\midrule
\multicolumn{8}{r}{\emph{continued on the next page}}\\
\endfoot
\bottomrule
\endlastfoot
$b^{(3)}_{1}$ & $(3,31,5,5,5)$ & $b^{(3)}_{2}$ & $(3,5,31,5,5)$ & $b^{(3)}_{3}$ & $(3,5,5,31,5)$ & $b^{(3)}_{4}$ & $(3,5,5,5,31)$\\
$b^{(3)}_{5}$ & $(3,5,5,7,29)$ & $b^{(3)}_{6}$ & $(3,5,7,13,21)$ & $b^{(3)}_{7}$ & $(3,5,7,29,5)$ & $b^{(3)}_{8}$ & $(3,5,7,5,29)$\\
$b^{(3)}_{9}$ & $(3,7,13,21,5)$ & $b^{(3)}_{10}$ & $(3,7,13,5,21)$ & $b^{(3)}_{11}$ & $(3,7,29,5,5)$ & $b^{(3)}_{12}$ & $(3,7,5,13,21)$\\
$b^{(3)}_{13}$ & $(3,7,5,29,5)$ & $b^{(3)}_{14}$ & $(3,7,5,5,29)$ & $b^{(3)}_{15}$ & $(31,3,5,5,5)$ & $b^{(3)}_{16}$ & $(7,11,21,5,5)$\\
$b^{(3)}_{17}$ & $(7,11,5,21,5)$ & $b^{(3)}_{18}$ & $(7,11,5,5,21)$ & $b^{(3)}_{19}$ & $(7,27,5,5,5)$ & $b^{(3)}_{20}$ & $(7,3,13,21,5)$\\
$b^{(3)}_{21}$ & $(7,3,13,5,21)$ & $b^{(3)}_{22}$ & $(7,3,29,5,5)$ & $b^{(3)}_{23}$ & $(7,3,5,13,21)$ & $b^{(3)}_{24}$ & $(7,3,5,29,5)$\\
$b^{(3)}_{25}$ & $(7,3,5,5,29)$ &  &  &  &  &  & \\
\end{longtable}
\endgroup

\subsection{Weight \texorpdfstring{\(\wt=(5,2,4,3)\)}{(5,2,4,3)}: dimension \texorpdfstring{\(5\)}{5}}\label{app:weight-4}

The following exponent vectors specify the complete basis of \((\QP_5)(5,2,4,3)\).

\begingroup
\scriptsize
\renewcommand{\arraystretch}{0.92}
\setlength{\LTleft}{0pt}
\setlength{\LTright}{0pt}
\begin{longtable}{@{}r@{\;}c@{\hspace{1.0em}}r@{\;}c@{\hspace{1.0em}}r@{\;}c@{\hspace{1.0em}}r@{\;}c@{}}
\toprule
index & \(A\) & index & \(A\) & index & \(A\) & index & \(A\)\\
\midrule
\endfirsthead
\toprule
index & \(A\) & index & \(A\) & index & \(A\) & index & \(A\)\\
\midrule
\endhead
\midrule
\multicolumn{8}{r}{\emph{continued on the next page}}\\
\endfoot
\bottomrule
\endlastfoot
$b^{(4)}_{1}$ & $(3,7,13,13,13)$ & $b^{(4)}_{2}$ & $(7,11,13,13,5)$ & $b^{(4)}_{3}$ & $(7,11,13,5,13)$ & $b^{(4)}_{4}$ & $(7,11,5,13,13)$\\
$b^{(4)}_{5}$ & $(7,3,13,13,13)$ &  &  &  &  &  & \\
\end{longtable}
\endgroup

\subsection{Weight \texorpdfstring{\(\wt=(5,4,3,1,1)\)}{(5,4,3,1,1)}: dimension \texorpdfstring{\(480\)}{480}}\label{app:weight-5}

The following exponent vectors specify the complete basis of \((\QP_5)(5,4,3,1,1)\).

\begingroup
\scriptsize
\renewcommand{\arraystretch}{0.92}
\setlength{\LTleft}{0pt}
\setlength{\LTright}{0pt}
\begin{longtable}{@{}r@{\;}c@{\hspace{1.0em}}r@{\;}c@{\hspace{1.0em}}r@{\;}c@{\hspace{1.0em}}r@{\;}c@{}}
\toprule
index & \(A\) & index & \(A\) & index & \(A\) & index & \(A\)\\
\midrule
\endfirsthead
\toprule
index & \(A\) & index & \(A\) & index & \(A\) & index & \(A\)\\
\midrule
\endhead
\midrule
\multicolumn{8}{r}{\emph{continued on the next page}}\\
\endfoot
\bottomrule
\endlastfoot
$b^{(5)}_{1}$ & $(1,15,19,7,7)$ & $b^{(5)}_{2}$ & $(1,15,23,3,7)$ & $b^{(5)}_{3}$ & $(1,15,23,7,3)$ & $b^{(5)}_{4}$ & $(1,15,3,23,7)$\\
$b^{(5)}_{5}$ & $(1,15,3,7,23)$ & $b^{(5)}_{6}$ & $(1,15,7,19,7)$ & $b^{(5)}_{7}$ & $(1,15,7,23,3)$ & $b^{(5)}_{8}$ & $(1,15,7,3,23)$\\
$b^{(5)}_{9}$ & $(1,15,7,7,19)$ & $b^{(5)}_{10}$ & $(1,3,15,23,7)$ & $b^{(5)}_{11}$ & $(1,3,15,7,23)$ & $b^{(5)}_{12}$ & $(1,3,31,7,7)$\\
$b^{(5)}_{13}$ & $(1,3,7,15,23)$ & $b^{(5)}_{14}$ & $(1,3,7,31,7)$ & $b^{(5)}_{15}$ & $(1,3,7,7,31)$ & $b^{(5)}_{16}$ & $(1,31,3,7,7)$\\
$b^{(5)}_{17}$ & $(1,31,7,3,7)$ & $b^{(5)}_{18}$ & $(1,31,7,7,3)$ & $b^{(5)}_{19}$ & $(1,7,11,23,7)$ & $b^{(5)}_{20}$ & $(1,7,11,7,23)$\\
$b^{(5)}_{21}$ & $(1,7,15,19,7)$ & $b^{(5)}_{22}$ & $(1,7,15,23,3)$ & $b^{(5)}_{23}$ & $(1,7,15,3,23)$ & $b^{(5)}_{24}$ & $(1,7,15,7,19)$\\
$b^{(5)}_{25}$ & $(1,7,27,7,7)$ & $b^{(5)}_{26}$ & $(1,7,3,15,23)$ & $b^{(5)}_{27}$ & $(1,7,3,31,7)$ & $b^{(5)}_{28}$ & $(1,7,3,7,31)$\\
$b^{(5)}_{29}$ & $(1,7,31,3,7)$ & $b^{(5)}_{30}$ & $(1,7,31,7,3)$ & $b^{(5)}_{31}$ & $(1,7,7,11,23)$ & $b^{(5)}_{32}$ & $(1,7,7,15,19)$\\
\end{longtable}
\endgroup

\subsection{Weight \texorpdfstring{\(\wt=(5,4,3,3)\)}{(5,4,3,3)}: dimension \texorpdfstring{\(175\)}{175}}\label{app:weight-6}

The following exponent vectors specify the complete basis of \((\QP_5)(5,4,3,3)\).

\begingroup
\scriptsize
\renewcommand{\arraystretch}{0.92}
\setlength{\LTleft}{0pt}
\setlength{\LTright}{0pt}
\begin{longtable}{@{}r@{\;}c@{\hspace{1.0em}}r@{\;}c@{\hspace{1.0em}}r@{\;}c@{\hspace{1.0em}}r@{\;}c@{}}
\toprule
index & \(A\) & index & \(A\) & index & \(A\) & index & \(A\)\\
\midrule
\endfirsthead
\toprule
index & \(A\) & index & \(A\) & index & \(A\) & index & \(A\)\\
\midrule
\endhead
\midrule
\multicolumn{8}{r}{\emph{continued on the next page}}\\
\endfoot
\bottomrule
\endlastfoot
$b^{(6)}_{1}$ & $(1,15,15,15,3)$ & $b^{(6)}_{2}$ & $(1,15,15,3,15)$ & $b^{(6)}_{3}$ & $(1,15,15,7,11)$ & $b^{(6)}_{4}$ & $(1,15,3,15,15)$\\
$b^{(6)}_{5}$ & $(1,15,7,11,15)$ & $b^{(6)}_{6}$ & $(1,15,7,15,11)$ & $b^{(6)}_{7}$ & $(1,3,15,15,15)$ & $b^{(6)}_{8}$ & $(1,7,11,15,15)$\\
$b^{(6)}_{9}$ & $(1,7,15,11,15)$ & $b^{(6)}_{10}$ & $(1,7,15,15,11)$ & $b^{(6)}_{11}$ & $(15,1,15,15,3)$ & $b^{(6)}_{12}$ & $(15,1,15,3,15)$\\
$b^{(6)}_{13}$ & $(15,1,15,7,11)$ & $b^{(6)}_{14}$ & $(15,1,3,15,15)$ & $b^{(6)}_{15}$ & $(15,1,7,11,15)$ & $b^{(6)}_{16}$ & $(15,1,7,15,11)$\\
$b^{(6)}_{17}$ & $(15,15,1,15,3)$ & $b^{(6)}_{18}$ & $(15,15,1,3,15)$ & $b^{(6)}_{19}$ & $(15,15,1,7,11)$ & $b^{(6)}_{20}$ & $(15,15,15,1,3)$\\
$b^{(6)}_{21}$ & $(15,15,15,3,1)$ & $b^{(6)}_{22}$ & $(15,15,3,1,15)$ & $b^{(6)}_{23}$ & $(15,15,3,13,3)$ & $b^{(6)}_{24}$ & $(15,15,3,15,1)$\\
$b^{(6)}_{25}$ & $(15,15,3,3,13)$ & $b^{(6)}_{26}$ & $(15,15,3,5,11)$ & $b^{(6)}_{27}$ & $(15,15,3,7,9)$ & $b^{(6)}_{28}$ & $(15,15,7,1,11)$\\
$b^{(6)}_{29}$ & $(15,15,7,11,1)$ & $b^{(6)}_{30}$ & $(15,15,7,3,9)$ & $b^{(6)}_{31}$ & $(15,15,7,9,3)$ & $b^{(6)}_{32}$ & $(15,3,1,15,15)$\\
\end{longtable}
\endgroup

\subsection{The invariant polynomial \texorpdfstring{\(\zeta\)}{zeta}}\label{app:zeta}

The invariant representative has \(283\) distinct terms, all belonging to the first weight block.  With the indexing of Appendix~\ref{app:weight-1}, its exact value is
\[
\zeta=\sum_{j\in I_\zeta} b_j^{(1)},
\]
where the \(283\)-element support \(I_\zeta\subset\{1,\ldots,1891\}\) is listed below in increasing order:

\begin{center}
\begingroup
\scriptsize
\setlength{\tabcolsep}{4.4pt}
\renewcommand{\arraystretch}{0.95}
\begin{tabular}{*{12}{r}}
3 & 6 & 8 & 11 & 12 & 16 & 19 & 20 & 21 & 28 & 29 & 30\\
31 & 32 & 38 & 39 & 41 & 42 & 80 & 82 & 229 & 232 & 234 & 237\\
238 & 240 & 247 & 250 & 254 & 258 & 259 & 260 & 264 & 267 & 272 & 279\\
280 & 281 & 282 & 283 & 284 & 286 & 288 & 292 & 295 & 297 & 309 & 310\\
312 & 315 & 321 & 324 & 325 & 328 & 329 & 331 & 332 & 358 & 359 & 360\\
375 & 377 & 394 & 401 & 402 & 403 & 407 & 412 & 413 & 416 & 419 & 422\\
423 & 436 & 439 & 440 & 442 & 449 & 452 & 481 & 483 & 488 & 490 & 491\\
493 & 722 & 725 & 728 & 730 & 733 & 734 & 736 & 738 & 740 & 746 & 753\\
754 & 755 & 756 & 763 & 768 & 775 & 776 & 777 & 779 & 780 & 791 & 793\\
798 & 805 & 806 & 808 & 821 & 822 & 825 & 827 & 833 & 834 & 836 & 865\\
866 & 867 & 868 & 870 & 873 & 896 & 897 & 900 & 903 & 931 & 932 & 933\\
948 & 950 & 969 & 970 & 971 & 972 & 976 & 980 & 981 & 983 & 986 & 991\\
992 & 995 & 997 & 999 & 1004 & 1009 & 1014 & 1018 & 1022 & 1024 & 1026 & 1027\\
1028 & 1029 & 1036 & 1038 & 1049 & 1050 & 1058 & 1060 & 1061 & 1062 & 1064 & 1065\\
1066 & 1067 & 1068 & 1070 & 1071 & 1072 & 1073 & 1075 & 1093 & 1096 & 1105 & 1106\\
1118 & 1119 & 1120 & 1130 & 1132 & 1137 & 1139 & 1140 & 1141 & 1142 & 1148 & 1149\\
1157 & 1158 & 1159 & 1160 & 1161 & 1162 & 1163 & 1166 & 1167 & 1168 & 1177 & 1178\\
1179 & 1180 & 1182 & 1191 & 1192 & 1196 & 1198 & 1206 & 1207 & 1213 & 1214 & 1215\\
1216 & 1217 & 1218 & 1220 & 1222 & 1223 & 1224 & 1225 & 1226 & 1229 & 1296 & 1297\\
1298 & 1313 & 1315 & 1337 & 1339 & 1350 & 1351 & 1387 & 1389 & 1390 & 1391 & 1392\\
1394 & 1401 & 1404 & 1433 & 1435 & 1440 & 1442 & 1443 & 1445 & 1473 & 1476 & 1488\\
1489 & 1520 & 1522 & 1533 & 1534 & 1543 & 1545 & 1546 & 1548 & 1549 & 1556 & 1557\\
1561 & 1563 & 1571 & 1572 & 1580 & 1581 & 1582 & 1583 & 1585 & 1587 & 1588 & 1589\\
1590 & 1591 & 1594 & 1622 & 1623 & 1629 & 1630 &  &  &  &  & \\
\end{tabular}
\endgroup
\end{center}

Because every \(b_j^{(1)}\) is specified by its exponent vector in Appendix~\ref{app:weight-1}, the support table is a complete term-by-term specification of \(\zeta\); all coefficients are \(1\in\F_2\), and no monomial is repeated.

\end{document}